\documentclass[12pt,a4paper,reqno]{amsart}
\usepackage{amssymb, amscd}

\newtheorem{thm}{Theorem}[section]

\newtheorem*{thm*}{Theorem}

\newtheorem{lem}[thm]{Lemma}

\newtheorem*{remark*}{Remarks}

\newtheorem*{defn*}{Definition}
\newtheorem*{claim*}{Claim}

\theoremstyle{definition}

 \renewcommand{\sectionmark}[1]{}

\begin{document}

\title[Complete description of rational points of Diophantine equation $ X^4 + Y^4 = Z^4 +  W^4 $] 
{Complete description of rational points of Diophantine equation $ X^4  + Y^4 = Z^4  + W^4 $}
\author[M. A. Reynya]{M. A. Reynya}
\address{Ahva St. 15/14, Haifa, Israel}

\email{misha\_371@mail.ru}


\begin{abstract}
In this paper we consider Diophantine equation $ x^4 + y^4 = z^4 + w^4 $ 

(1)We construct some family of cubic curves.We prove that every rational point on  Quartica $ x^4 + y^4 = z^4 + w^4 $
 can be mapped to a point on some curve of this family. We also prove the opposite: each rational point belonging to our family of curves can be mapped to a rational point on the Quartica.
 
(2) We find the point on our family of curves corresponding to a parametric solution of Leonard Euler. We construct several new parametric solutions of our Quartica, using a parametric solution of Leonard Euler and the algebraic operation
on the cubic curves.

(3)We present an algorithm to find all rational points on our Quartica.

\end{abstract}

\maketitle

\baselineskip 20pt
\section {Introduction} \label{sec:intro}

It is known that the Quartica $X^4 + Y^4 = Z^4 + W^4$ has infinitely many integer solutions. Leonard Euler
found the first parametric solution. It was proved in the (see \cite{Bi1}  by Swinnerton-Dyer that this diophantine equation has infinitely
many rational parametric solitions.The following authors have found various parametric solutions: (see \cite{Bi2},\cite{Hr}.
The aim of this paper is to give algorithm finding  all solutions of this Quartica.In the Section $2$ we 
show that this Quartica is equivalent to a system of two equations of the third degree in six variables.We prove that every rational point on  Quartica 
can be mapped to a rational point which is solution to this system of two equations, and the opposite: each rational point that is solution to a system of two equations can be mapped to a rational point on the Quartica. Our Quartica has  elementary solutions : $(m,n,m,n)$ and $(s,t,-t,s)$. And this means that these elementary solutions are mapped to solutions of a system of equations we constructed. So we got two solutions of our system of equations. Next using these two solutions we find a new solution of
our system of equations,as a result we get the expression depending on the 5 variables : $(m , n , r , s , t)$.This expression is the cubic   in variables $m$ , $n$ , $r$ with coefficients depending on variables $s$ and $t$.So we construct a family of cubic curves.
We prove that this family of cubic curves contains images of all rational points of our Quartica. As an example we take the Quartca point $( 59 , 158 , 133 , 134 )$ and build
cubic curve of our family which contains the image of this point. We also prove the opposite: each rational point belonging to our family of curves is mapped to a rational point on the Quartica.

In the section (3) we build several new parametric solutions of our Quartica,using a parametric solution of Leonard Euler.
At the beginning we find a point on our family of cubic curves  depending on the parameters s and t.Then we find the image 
of this point on our Quartica.That is, we get a parametric solution to our Quartica. This is the parametric solution of Leonard Euler.In accordance with our costruction every solution on Quartica  can be associated to a new solution.We call it the pair solution.
Using Mathematika 7 + we find a parametric solution to pair solution of Leonard Euler.
Because every point on the Quartica has image on the cubic curve and opposite, and on the cubic curve  there exist
associative operation we can  construct new  solutions using solution of Leonard Euler.

The process looks like this :

1.We find the image of solution of Leonard Euler on the cubic, let us denote this point $E$

2.Using the  tangent method we build the point $E\cdot E$ on the cubic.

3.We find the image of the point $E\cdot E$ on the Quartica - the new solution.

As an example we take the point $E$ - $( 59 , 158 , 133 , 134 )$.And we get 

the point $E\cdot E$ - $((-8450072351) , 520471467675 , 487934246375 , 59481958899))$, and 
 
 $((-3535404127283) , (-132758926000) , (3343735015475) , (-2363831080408))$ - the pair point.
 
 Using Mathematika 7 + we find an appropriate parametric solution  $E\cdot E$ and its pair.

 In section 4 we present the algorithm for finding all rational points on the our Quartica.


\section {Construction the family of elliptic curves} \label{sec:quin}

\begin{lem} \label{th:repr}
 The Quartica $x^4 + y^4 = z^4 + w^4$ then, and only then  has a rational point (x,y,z,w), if the system equations
$\left\{\begin{array}{clrr} a\cdot x^2-b\cdot y^2=a\cdot z^2+b\cdot w^2, \\
 b\cdot x^2+a\cdot y^2=-b\cdot z^2+a\cdot w^2\end{array} \right.$.
 has the rational point (a,b,x,y,z,w).
\end{lem}
\begin{proof}
To solve the system equations we receive: $a\cdot (x^2-z^2)=b\cdot (y^2+w^2)$,
$b\cdot (x^2+z^2)=a\cdot (w^2-y^2)$,and $a/b=(y^2+w^2)/(x^2-z^2)=(x^2+z^2)/(w^2-y^2)$,
so $x^4+y^4=w^4+z^4$.
And opposite: $x^4+y^4=w^4+z^4$,$w^4-y^4=x^4-z^4$,$(w^2-y^2)\cdot (w^2+y^2)=(x^2-z^2)\cdot (x^2+z^2)$,$(y^2+w^2)/(x^2-z^2)=(x^2+z^2)/(w^2-y^2)$,we will suppose that:$(y^2+w^2)/(x^2-z^2)=(x^2+z^2)/(w^2-y^2)=a/b$,
than $(y^2+w^2)/(x^2-z^2)=a/b$,and,$(x^2+z^2)/(w^2-y^2)=a/b$, and we get the initial system equations:
$\left\{\begin{array}{clrr} a\cdot x^2-b\cdot y^2=a\cdot z^2+b\cdot w^2, \\
 b\cdot x^2+a\cdot y^2=-b\cdot z^2+a\cdot w^2\end{array} \right.$.
\end{proof}

\begin{thm} \label{th:general} 

The system of equations:

$\left\{\begin{array}{clrr} a\cdot x^2-b\cdot y^2=a\cdot z^2+b\cdot w^2, \\
 b\cdot x^2+a\cdot y^2=-b\cdot z^2+a\cdot w^2\end{array} \right.$.

has two parametric solutions :

${x=m,y=n,z=m,w=n,b=0,a=r}$,

${x=s,y=t,z=-t,w=s,b=1,a=(s^2+t^2)/(s^2-t^2)}$

\end{thm}

\begin{proof}

Points:

${x=m,y=n,z=m,w=n}$

${x=s,y=t,z=-t,w=s}$

are solutions of Quartica $x^4 + y^4 = z^4 + w^4$,therefore for Lemma 2.1
there must be a solution for system of equations:
$\left\{\begin{array}{clrr} a\cdot x^2-b\cdot y^2=a\cdot z^2+b\cdot w^2, \\
 b\cdot x^2+a\cdot y^2=-b\cdot z^2+a\cdot w^2\end{array} \right.$.
To solve this system for $a$ and $b$ for every point we can to prove the Theorem 2.2.

(1)
$\left\{\begin{array}{clrr} a\cdot m^2-b\cdot n^2=a\cdot m^2+b\cdot n^2, \\
 b\cdot m^2+a\cdot n^2=-b\cdot m^2+a\cdot n^2\end{array} \right.$.
 
$a=r , b=0$.

(2)
$\left\{\begin{array}{clrr} a\cdot s^2-b\cdot t^2=a\cdot t^2+b\cdot s^2, \\
 b\cdot s^2+a\cdot t^2=-b\cdot t^2+a\cdot s^2\end{array} \right.$.
  
$a=(s^2+t^2)/(s^2-t^2), b=1$.

\end {proof}

Now we have a system of two equations of the third degree,and two solutions for this system.

\begin{thm} \label{th:general} 

The system of equations:

$\left\{\begin{array}{clrr} a\cdot x^2-b\cdot y^2=a\cdot z^2+b\cdot w^2, \\
 b\cdot x^2+a\cdot y^2=-b\cdot z^2+a\cdot w^2\end{array} \right.$.

has  parametric solution:

$x=(m\cdot g + s),y=(n\cdot g + t),z=(m\cdot g - t),w=(n\cdot g + s),b=(0\cdot g + 1),a=(r\cdot g + (s^2+t^2)/(s^2-t^2))$,

where

 $g = (2\cdot m\cdot s^2-2\cdot n\cdot s^2+r\cdot s^3-r\cdot s^2\cdot t+2\cdot m\cdot t^2+2\cdot n\cdot t^2-r\cdot s\cdot t^2+r\cdot t^3)/(2\cdot (s-t)\cdot (n^2-m\cdot r\cdot s - m\cdot r \cdot t))$ from the first equation of the system         ,

and

 $g = (-2\cdot m\cdot s^2+2\cdot n\cdot s^2+r\cdot s^3+r\cdot s^2\cdot t+2\cdot m\cdot t^2+2\cdot n\cdot t^2-r\cdot s\cdot t^2-r\cdot t^3)/(2\cdot (s+t)\cdot (m^2-n\cdot r\cdot s + n\cdot r \cdot t))$ from the second equation of the system

\end{thm}

\begin{proof}

Substitute the expressions:

$x=(m\cdot g + s),y=(n\cdot g + t),z=(m\cdot g - t),w=(n\cdot g + s),b=(0\cdot g + 1),a=(r\cdot g + (s^2+t^2)/(s^2-t^2))$,

in each  of the equations of our system.

Because   $(x=m,y=n,z=m,w=n,b=0,a=r)$ and

$(x=s,y=t,z=-t,w=s,b=1,a=(s^2+t^2)/(s^2-t^2))$

are solutions to each of the equations ,we receive two linear equations for $g$,because $g=0$ and $g=\infty$ are the roots
of each of the equations.

If we compute value of $g$ for each of the equations, we get the expressions for $g$ specified in the condition 
of the theorem.

\end{proof}

Now it is clear that condition for the existence of solution for system of equations and for Lemma 2.1 for equation
$x^4+y^4=z^4+w^4$  is the equality of two expressions for $g$ from Theorema 2.3.

Consider the expression resulting from the equality of two expressions for $g$:

$(2\cdot m\cdot s^2-2\cdot n\cdot s^2+r\cdot s^3-r\cdot s^2\cdot t+2\cdot m\cdot t^2+2\cdot n\cdot t^2-r\cdot s\cdot t^2+r\cdot t^3)/(2\cdot (s-t)\cdot (n^2-m\cdot r\cdot s - m\cdot r \cdot t))$ = $(-2\cdot m\cdot s^2+2\cdot n\cdot s^2+r\cdot s^3+r\cdot s^2\cdot t+2\cdot m\cdot t^2+2\cdot n\cdot t^2-r\cdot s\cdot t^2-r\cdot t^3)/(2\cdot (s+t)\cdot (m^2-n\cdot r\cdot s + n\cdot r \cdot t))$,

or 

$(2\cdot m\cdot s^2-2\cdot n\cdot s^2+r\cdot s^3-r\cdot s^2\cdot t+2\cdot m\cdot t^2+2\cdot n\cdot t^2-r\cdot s\cdot t^2+r\cdot t^3)\cdot (2\cdot (s+t)\cdot (m^2-n\cdot r\cdot s + n\cdot r \cdot t))$ = $(-2\cdot m\cdot s^2+2\cdot n\cdot s^2+r\cdot s^3+r\cdot s^2\cdot t+2\cdot m\cdot t^2+2\cdot n\cdot t^2-r\cdot s\cdot t^2-r\cdot t^3)\cdot (2\cdot (s-t)\cdot (n^2-m\cdot r\cdot s - m\cdot r \cdot t))$.

This expression is the cubic in variables $m$ , $n$ , $r$ with coefficients depending on variables $s$ and $t$.

\begin{thm} \label{th:general} 
For any solution $x$,$y$,$z$,$w$
of the equation $x^4+y^4=z^4+w^4$ there are such variable values $m$ , $n$ , $r$ , $s$ , $t$ for the given values $s$ and $t$
variable values $m$ , $n$ , $r$ are the solution of our cubic.

\end{thm}

\begin{proof}
Suppose that ($x$,$y$,$z$,$w$) the rational point of equation, than for Lemma 2.1 there exist $a$ and $b$ so that ($x$,$y$,$z$,$w$,$a$ $b$) the rational point of system equations.

Consider the system equations:$x=(m\cdot g + s),y=(n\cdot g + t),z=(m\cdot g - t),w=(n\cdot g + s)$
Suppose that $g=1$,The discriminant of this system is not equal to $0$,therefore $m,n,s,t$ exist.
But  the points ($m$,$n$,$m$,$n$) and ($s$,$t$,$-t$,$s$) are the solutions of Quartica $x^4+y^4=z^4+w^4$,and therefore
the points :

(${m,n,m,n,0,r}$),

(${s,t,-t,s,1,(s^2+t^2)/(s^2-t^2)}$)

are the solutions of system of equations.Now in in accordance with our construction the
$x,y,z,w,b=1,a=(x^2+z^2)/(w^2-y^2)$ are the solution of our system  ,on the other hand because there are such $m,n,s,t$ that
$x=(m\cdot g + s),y=(n\cdot g + t),z=(m\cdot g - t),w=(n\cdot g + s)$ for $g=1$, the point $(x=(m\cdot g + s),y=(n\cdot g + t),z=(m\cdot g - t),w=(n\cdot g + s),b=0\cdot g + 1,a=r\cdot g + (s^2+t^2)/(s^2-t^2))$ is the solution of the system if $g=1$
and $r=(x^2+z^2)/(w^2-y^2)-(s^2+t^2)/(s^2-t^2)$,and therefore for any solution $x$,$y$,$z$,$w$
of the equation $x^4+y^4=z^4+w^4$ there are such variable values $m$ , $n$ , $r$ , $s$ , $t$ for the given values $s$ and $t$
variable values $m$ , $n$ , $r$ are the solution of our cubic.

\end{proof}

EXAMPLE.

Now we have solution of Quartica $x^4+y^4=z^4+w^4$ : ($x=59$ , $y=158$ , $z=133$ , $w=134$).
Show that there is a  cubic curve to which this point belongs.

Consider the system of equations:$(m\cdot g + s)=59,y=(n\cdot g + t)=158,z=(m\cdot g - t)=133,w=(n\cdot g + s)=134$.
Suppose $g=1$,then $m=108, n=183, s=-49, t=-25$.
    
If we substitute into the equation of the cubic the values  of $s$ and $t$ we receive:

$(-1513\cdot m + 888\cdot n + 10656\cdot r)\cdot (37\cdot (m^2 + 24\cdot n \cdot r) = (888\cdot m - 1513\cdot n + 32856\cdot r)\cdot (12\cdot (n^2 + 74\cdot m\cdot r))$

If we  substitute in this cubic  values $m=108, n=183$ we receive for $r$ equation of second degree ,which must have one rational root of the $g=1$. Since the equation is of second degree there is another rational root corresponding to the new solution.

Equation:$72339650100 + 8604986400\cdot r - 1419379200\cdot r^2 = 0$ has two rational roots:

$r1 = -2797/592 , r2 = 3193/296$

We substitute values for $m , n , s , t$ to Quartica and receive:

$(108\cdot g + (-49))^4 + (183\cdot g + (-25))^4 - (108\cdot g - (-25))^4 - (183\cdot g + (-49))^4 = 0$,or

$17107200\cdot g - 232567200\cdot g^2 + 215460000\cdot g^3 = 0$
      
This equation has two rational roots :the root $g=1$ corresponding to the point ($x=59$ , $y=158$ , $z=133$ , $w=134$) on the
Quartica and the point($m=108,n=183,r=-2797/592$)on the cubic.And the root $g_1=3193/296$,corresponding to some new point on the 
Quartica:($x_1=-134413,y_1=-34813,z_1=111637,w_1=-114613$),and the point ($m_1=108,n_1=183,r_1=3193/296$) on the cubic.

So picking an arbitrary rational point ($x=59$ , $y=158$ , $z=133$ , $w=134$) on the Quartica we found  the cubic to which  belong the rational point ($m=108,n=183,r=-2797/592$), where  ($m,n,r$) are functions from ($x,y,z,w$) and found  one more point on the Quartica ($x_1=-134413,y_1=-34813,z_1=111637,w_1=-114613$),corresponding to the point ($m_1=108,n_1=183,r_1=3193/296$) on the same cubic.

So the family of cubics is:

$(2\cdot m\cdot s^2-2\cdot n\cdot s^2+r\cdot s^3-r\cdot s^2\cdot t+2\cdot m\cdot t^2+2\cdot n\cdot t^2-r\cdot s\cdot t^2+r\cdot t^3)\cdot (2\cdot (s+t)\cdot (m^2-n\cdot r\cdot s + n\cdot r \cdot t))$ = $(-2\cdot m\cdot s^2+2\cdot n\cdot s^2+r\cdot s^3+r\cdot s^2\cdot t+2\cdot m\cdot t^2+2\cdot n\cdot t^2-r\cdot s\cdot t^2-r\cdot t^3)\cdot (2\cdot (s-t)\cdot (n^2-m\cdot r\cdot s - m\cdot r \cdot t))$.
Since each point in our family of cubic curves match point-solution of cubic equations then by Lemma 2 the following is true: every point in our family of cubic curves match point
on the surface
containing images of all rational points of our Quartica.Since each point in our family of cubic curves match point-solution of cubic equations then Lemma 2 the following is true: every point in our family of cubic curves match point
on the Quartica.

\section {Construction of parametric solutions of Quartica $X^4 + Y^4 = Z^4 + W^4$} \label{sec:gen}

Consider the family of cubics constructed in the previous chapter:


$(2\cdot m\cdot s^2-2\cdot n\cdot s^2+r\cdot s^3-r\cdot s^2\cdot t+2\cdot m\cdot t^2+2\cdot n\cdot t^2-r\cdot s\cdot t^2+r\cdot t^3)\cdot (2\cdot (s+t)\cdot (m^2-n\cdot r\cdot s + n\cdot r \cdot t))$ = $(-2\cdot m\cdot s^2+2\cdot n\cdot s^2+r\cdot s^3+r\cdot s^2\cdot t+2\cdot m\cdot t^2+2\cdot n\cdot t^2-r\cdot s\cdot t^2-r\cdot t^3)\cdot (2\cdot (s-t)\cdot (n^2-m\cdot r\cdot s - m\cdot r \cdot t))$.

\begin{thm} \label{th:general} 
This family of cubics  contains the point: ($m = (-(( s^5 - s^3\cdot t^2 - s^2\cdot t^3 + t^5)/(4\cdot s^2\cdot t^2))), n = (-(( 
   s^5 - s^3\cdot t^2 + s^2\cdot t^3 - t^5)/(4\cdot s^2\cdot t^2))), r = 1$)
   
\end{thm}

\begin{proof}
To prove the theorem it is enough to equate to 0 two expressions in the left and right side of the equation for a family of curves.
We get a system of equations:
 
$\left\{\begin{array}{clrr} (2\cdot m\cdot s^2-2\cdot n\cdot s^2+r\cdot s^3-r\cdot s^2\cdot t+2\cdot m\cdot t^2+2\cdot n\cdot t^2-r\cdot s\cdot t^2+r\cdot t^3)=0, \\
(-2\cdot m\cdot s^2+2\cdot n\cdot s^2+r\cdot s^3+r\cdot s^2\cdot t+2\cdot m\cdot t^2+2\cdot n\cdot t^2-r\cdot s\cdot t^2-r\cdot t^3)=0\end{array} \right.$.

Solving this system, provided that the $r=1$,we prove the theorem

\end{proof}

Now we have the point on the family of cubics.Now construct a point on Quartica.For this we will solve the equation:

$(m\cdot g+s)^4+(n\cdot g+t)^4-(m\cdot g-t)^4-(n\cdot g +s)^4 = 0$

Or substituting the values for $m$ and $n$:

$((-(( s^5 - s^3\cdot t^2 - s^2\cdot t^3 + t^5)/(4\cdot s^2\cdot t^2)))\cdot g+s)^4+((-(( 
   s^5 - s^3\cdot t^2 + s^2\cdot t^3 - t^5)/(4\cdot s^2\cdot t^2)))\cdot g+t)^4-((-(( s^5 - s^3\cdot t^2 - s^2\cdot t^3 + t^5)/(4\cdot s^2\cdot t^2)))\cdot g-t)^4-((-(( 
   s^5 - s^3\cdot t^2 + s^2\cdot t^3 - t^5)/(4\cdot s^2\cdot t^2)))\cdot g +s)^4 = 0$

Using the program Mathematika 7+ we receive:

$g = (-((12\cdot s^4\cdot t^4)/((s^2 - t^2)\cdot (s^6 - 2\cdot s^4\cdot t^2 - 2\cdot s^2\cdot t^4 + t^6))))$

And, accordingly:

$x = s\cdot (s^6 + s^4\cdot t^2 - 2\cdot s^2\cdot t^4 - 3\cdot s\cdot t^5 + t^6)$

$y = t\cdot (s^6 + 3\cdot s^5\cdot t - 2\cdot s^4\cdot t^2 + s^2\cdot t^4 + t^6)$

$z = -t\cdot (s^6 - 3\cdot s^5\cdot t - 2\cdot s^4\cdot t^2 + s^2\cdot t^4 + t^6)$

$w = s\cdot (s^6 + s^4\cdot t^2 - 2\cdot s^2\cdot t^4 + 3\cdot s\cdot t^5 + t^6)$

This is the solution of Leonard Euler

In an example considered above  $ 59,158,133,134$ corresponds to the solution of Leonard Euler for $s=2$,$t=1$:
$(x=134,y=133,z=59,w=158)$.However, we got a new solution ($x_1=-134413,y_1=-34813,z_1=111637,w_1=-114613$)Call it a pair of the solution of Leonard Euler.To find the solution in parametric form we will carry out the following:

Solve the system of equations:

$\left\{\begin{array}{clrr} ((M + S) = -t\cdot (s^6 - 3\cdot s^5\cdot t - 2\cdot s^4\cdot t^2 + s^2\cdot t^4 + t^6), \\$
$(N + T) = s\cdot (s^6 + s^4\cdot t^2 - 2\cdot s^2\cdot t^4 + 3\cdot s\cdot t^5 + t^6), \\ (M - T) == t\cdot (s^6 + 3\cdot$ $s^5\cdot t - 2\cdot s^4\cdot t^2 + s^2\cdot t^4 + t^6), \\(N + S) == s\cdot (s^6 + s^4\cdot t^2 - 2\cdot s^2\cdot t^4 -$ $3\cdot s\cdot t^5 + t^6) \end{array} \right.$

Solving this system we get:

$M=3\cdot (s^5\cdot t^2 + s^2\cdot t^5), N=(s^7 + s^6\cdot t + s^5\cdot t^2 - 2\cdot s^4\cdot t^3 - $
$   2\cdot s^3\cdot t^4 + s^2\cdot t^5 + s\cdot t^6 + t^7), S=(-s^6\cdot t + 2\cdot s^4\cdot t^3 - $
$   4\cdot s^2\cdot t^5 - t^7), T=(-s^6\cdot t + 2\cdot s^4\cdot t^3 + 2\cdot s^2\cdot t^5 - t^7)$

Now we solve  equation:

$(M\cdot g+S)^4+(N\cdot g+T)^4-(M\cdot g-T)^4-(N\cdot+S)^4=0$

In accordance with our construction,this equation has the roots:

$g_1=0,g_2=\infty,g_3=1$,$g_3=1$ - corresponds to the solution of Euler,and $g_4$ we compute with Mathematika 7+:

$g_4=(-s^{13} + 2\cdot s^{12}\cdot t + 4\cdot s^{11}\cdot t^2 - 8\cdot s^{10}\cdot t^3 + 8\cdot s^7\cdot t^6 + 2\cdot s^6\cdot$ $t^7 - $
$ 18\cdot s^5\cdot t^8 + 18\cdot s^4\cdot t^9 - 14\cdot s^3\cdot t^{10} + 10\cdot s^2\cdot t^{11} - s\cdot t^{12} + $
$ 2\cdot t^{13})/(s^{13} + s^{12}\cdot t - 4\cdot s^{11}\cdot t^2 + 14\cdot s^{10}\cdot t^3 - 18\cdot s^9\cdot t^4 - $
$ 9\cdot s^8\cdot t^5 + 28\cdot s^7\cdot t^6 - 8\cdot s^6\cdot t^7 - 4\cdot s^3\cdot t^{10} + 5\cdot s^2\cdot t^{11} + $
$ s\cdot t^{12} + t^{13})$

Substituting the value of $g_4$  in each of the expressions:$(M\cdot g_4+S),(N\cdot g_4+T),(M\cdot g_4-T),(N\cdot g_4+S)$
we get a new parametric solution :

$x=(-t\cdot (s^{18} + 3\cdot s^{17}\cdot t - 15\cdot s^{16}\cdot t^2 + 15\cdot s^{15}\cdot t^3 + 6\cdot s^{14}\cdot t^4 -$ 
$    45\cdot s^{13}\cdot t^5 + 82\cdot s^{12}\cdot t^6 - 15\cdot s^{11}\cdot t^7 - 123\cdot s^{10}\cdot t^8 + $
    $        171\cdot s^9\cdot t^9 - 159\cdot s^8\cdot t^{10} + 159\cdot s^7\cdot t^{11} - 98\cdot s^6\cdot t^{12} + $
        $    30\cdot s^5\cdot t^{13} - 12\cdot s^4\cdot t^{14} + 3\cdot s^2\cdot t^{16} + t^{18}))$
    
 $  y= ((-s^{19} + s^{18}\cdot t + 3\cdot s^{17}\cdot t^2 + 3\cdot s^{16}\cdot t^3 - 21\cdot s^{15}\cdot t^4 +$
    $  12\cdot s^{14}\cdot t^5 + 44\cdot s^{13}\cdot t^6 - 86\cdot s^{12}\cdot t^7 + 93\cdot s^{11}\cdot t^8 - $
    $  87\cdot s^{10}\cdot t^9 - 3\cdot s^9\cdot t^{10} + 135\cdot s^8\cdot t^{11} - 142\cdot s^7\cdot t^{12} + $
    $  100\cdot s^6\cdot t^{13} - 72\cdot s^5\cdot t^{14} + 36\cdot s^4\cdot t^{15} - 12\cdot s^3\cdot t^{16} + $
    $  9\cdot s^2\cdot t^{17} - s\cdot t^{18} + t^{19}))$
   
 $ z= (t\cdot (s^{18} - 3\cdot s^{17}\cdot t + 3\cdot s^{16}\cdot t^2 + 21\cdot s^{15}\cdot t^3 - 60\cdot s^{14}\cdot t^4 + $
   $   27\cdot s^{13}\cdot t^5 + 58\cdot s^{12}\cdot t^6 - 75\cdot s^{11}\cdot t^7 + 57\cdot s^{10}\cdot t^8 -$ 
      $  63\cdot s^9\cdot t^9 + 63\cdot s^8\cdot t^{10} - 87\cdot s^7\cdot t^{11} + 100\cdot s^6\cdot t^{12} - $
      $  66\cdot s^5\cdot t^{13} + 36\cdot s^4\cdot t^{14} - 18\cdot s^3\cdot t^{15} + 9\cdot s^2\cdot t^{16} + t^{18}))$
        
 $w= ((-s^{19} + s^{18}\cdot t + 3\cdot s^{17}\cdot t^2 + 3\cdot s^{16}\cdot t^3 - 21\cdot s^{15}\cdot t^4 + $
    $6\cdot s^{14}\cdot t^5 + 44\cdot s^{13}\cdot t^6 - 62\cdot s^{12}\cdot t^7 - 15\cdot s^{11}\cdot t^8 + $
      $129\cdot s^{10}\cdot t^9 - 165 \cdot s^9\cdot t^{10} + 129\cdot s^8\cdot t^{11} - 88\cdot s^7\cdot t^{12} +$
       $46\cdot s^6\cdot t^{13} - 18\cdot s^5\cdot t^{14} + 6\cdot s^4\cdot t^{15} - 12\cdot s^3\cdot t^{16} + 3\cdot s^2\cdot $
       $ t^{17} - s\cdot t^{18} + t^{19}))$

Because every point on the Quartica has image on the cubic curve and opposite, and on the cubic curve  there exist
an algebraic operation we can  construct new  solutions,using solution of Leonard Euler.

The process looks like this :

1.We find the image of solution of Leonard Euler on the cubic,let us denote this point $E$

2.Using the  tangent method we build the point $E\cdot E$ on the cubic.

3.We find the image of the point $E\cdot E$ on the Quartica - the new solution.

Example.
We have the ratonal point on the Quartica : $59 , 158 , 133 , 134$

The cubic containg image of this rational point was constructed  in the 2 chapter of this paper.

This is the cubic : $(-1513\cdot m + 888\cdot n + 10656\cdot r)\cdot (37\cdot (m^2 + 24\cdot n \cdot r) = (888\cdot m -$ $1513\cdot n + 32856\cdot r)\cdot (12\cdot (n^2 + 74\cdot m\cdot r))$

The image point is: $m=108,n=183,r=(-2797/592)$(from chapter 2)

Now we can  use the tangent method for this cubic curve:

We substitude :$m=108$ to $M=(108\cdot g_1+k)$ , $n=183$ to $N=(183\cdot g_1+1)$ , $r=(-2797/592)$ to $R=((-2797/592)\cdot g_1+2)$ :

 $(-1513\cdot (108\cdot g_1+k) + 888\cdot (183\cdot g_1+1) + 10656\cdot ((-2797/592)\cdot g_1+2))\cdot (37\cdot ((108\cdot$ $g_1+k)^2 + 24\cdot (183\cdot g_1+1) \cdot ((-2797/592)\cdot g_1+2)) = (888\cdot (108\cdot g_1+k) - 1513\cdot (183\cdot g_1+1)$ $+ 32856\cdot r)\cdot (12\cdot ((183\cdot g_1+1)^2 + 74\cdot (108\cdot g_1+k)\cdot ((-2797/592)\cdot g_1+2)))$

So we receive:

$-((9\cdot g_1^2\cdot (-20155494924 + 562635949\cdot k) + $
$  72\cdot g_1 (308932483 - 21347816\cdot k + 570133\cdot k^2) +$ 
$  4\cdot (-38656812 + 116715168\cdot k + 755688\cdot k^2 + 55981\cdot k^3))/($
$ 3\cdot (-8 - 130800\cdot g_1 + 34164\cdot g_1^2 - 1184\cdot k + 2797\cdot g_1\cdot k) (672417\cdot g_1^2 - $
$    3\cdot g_1\cdot (213875 + 5328\cdot k) - 74\cdot (48 + k^2))))$

 We can choose the $k$ so that the coefficient for $g_1^2$ is equal to $0$:

$(-20155494924 + 562635949\cdot k) = 0$

Solving this equation we receive: $k = 20155494924/562635949$,and solving equation for $g_1$ :

$g_1 = (-3431129689319806/2216545393509777)$

Now we receive values from $M , N , R$:

$M = (-97052654280770532/738848464503259)$ , 

$N = (-208560062584004907/738848464503259)$,

$R = 6110629743471536675/656097436478893992$.

This is the image point on the cubic corresponding the new point on the Quartika.
 Now we can find this new point,solving the equation:  

$((-(97052654280770532/738848464503259))\cdot s + (-49))^4$

$+ ((-(208560062584004907/738848464503259))\cdot  s + (-25))^4$

$- ((-(97052654280770532/738848464503259))\cdot s - (-25))^4$

$- (-(208560062584004907/738848464503259))\cdot  s + (-49))^4 = 0$

This equation has two rational roots :

$s_1=(-15645116235856509325/187352780702663748309)$,and 

$s_2=(-258596962576140650/711525553297861767)$

This two roots correspond to the two points on the Quartica:

$(x_1=(-3535404127283) , y_1=(-132758926000) , z_1=(3343735015475) , w_1=(-2363831080408))$ , for $s_1$

$(-3535404127283)^4 + (-132758926000)^4 = 3343735015475^4 + (-2363831080408)^4$

and
  
$(x_2=(-8450072351) , y_2=520471467675 , z_2=487934246375 ,w_2=359481958899)$,for $s_2$

$(-8450072351)^4 + 520471467675^4 = 487934246375^4 + 359481958899^4$

Now we  return to our cubic :

This is the cubic : $(-1513\cdot m + 888\cdot n + 10656\cdot r)\cdot (37\cdot (m^2 + 24\cdot n \cdot r) = (888\cdot m - 1513\cdot n + 32856\cdot r)\cdot (12\cdot (n^2 + 74\cdot m\cdot r))$

We substitute the new values of $m=(-97052654280770532/738848464503259)$ ,

$n=(-208560062584004907/738848464503259)$, 
and receive  the equation of second degree for $r$ :

$-5260289575280440614252321193027166875 + $

$ 392359179683252386906081036910885000\cdot r + $
 
$ 18514574028136616634982730304244992\cdot r^2 = 0$

This equation has two rational roots :

$r_1 = (-860842465688650025/28219244579737376)$,

$r_2 = 6110629743471536675/656097436478893992$. 
  
The simple computing show that  $r_1 = (-860842465688650025/28219244579737376)$, correspond to solution

$(x_1=(-3535404127283) , y_1=(-132758926000) , z_1=(3343735015475) , w_1=(-2363831080408))$

and the $r_2 = 6110629743471536675/656097436478893992$, correspond to solution

$(x_2=(-8450072351) , y_2=520471467675 , z_2=487934246375 ,w_2=359481958899)$

So in accordance with our construction these are the points that we are seeking:

$(x_2=(-8450072351) , y_2=520471467675 , z_2=487934246375 ,w_2=359481958899)$

and 

$(x_1=(-3535404127283) , y_1=(-132758926000) , z_1=(3343735015475) , w_1=(-2363831080408))$ the pair point.

Using the program Wolfram Mathematika 7+  we can calculate the parametric solution and its pair parametric solution.

If $E$ is the point on the family of curves corresponding to parametric solution of Leonard Euler,the parametric solution 
corresponding to the point $E\cdot E$ is:

$x = (s^{13} + s^{12}\cdot t + 2\cdot s^{11}\cdot t^2 - 4\cdot s^{10}\cdot t^3 - 3\cdot s^9\cdot t^4 + 3\cdot s^8\cdot t^5 + 
     7\cdot s^7\cdot t^6 + 4\cdot s^6\cdot t^7 - 12\cdot s^5\cdot t^8 - 6\cdot s^4\cdot t^9 + 5\cdot s^3\cdot t^{10} - 
     s^2\cdot t^{11} + s\cdot t^{12} + t^{13})$ 
     
$y = (-s^{13} + s^{12}\cdot t + s^{11}\cdot t^2 + 
     5\cdot s^{10}\cdot t^3 + 6\cdot s^9\cdot t^4 - 12\cdot s^8\cdot t^5 - 4\cdot s^7\cdot t^6 + 7\cdot s^6\cdot t^7 - 
     3\cdot s^5\cdot t^8 - 3\cdot s^4\cdot t^9 + 4\cdot s^3\cdot t^{10} + 2\cdot s^2\cdot t^{11} - s\cdot t^{12} + 
     t^{13})$
     
$z =  (-(s^{13} + s^{12}\cdot t - s^{11}\cdot t^2 + 5\cdot s^{10}\cdot t^3 - 6\cdot s^9\cdot t^4 - 
       12\cdot s^8\cdot t^5 + 4\cdot s^7\cdot t^6 + 7\cdot s^6\cdot t^7 + 3\cdot s^5\cdot t^8 - 3\cdot s^4\cdot t^9 - 
       4\cdot s^3\cdot t^{10} + 2\cdot s^2\cdot t^{11} + s\cdot t^{12} + t^{13}))$
        
$w = (s^{13} - s^{12}\cdot t + 
     2\cdot s^{11}\cdot t^2 + 4\cdot s^{10}\cdot t^3 - 3\cdot s^9\cdot t^4 - 3\cdot s^8\cdot t^5 + 7\cdot s^7\cdot t^6 - 
     4\cdot s^6\cdot t^7 - 12\cdot s^5\cdot t^8 + 6\cdot s^4\cdot t^9 + 5\cdot s^3\cdot t^{10} + s^2\cdot t^{11} + 
     s\cdot t^{12} - t^{13})$

And the pair parametric solution :

$x = t\cdot (2\cdot s^{18} - 6\cdot s^{17}\cdot t - 3\cdot s^{16}\cdot t^2 - 3\cdot s^{15}\cdot t^3 - 9\cdot s^{14}\cdot t^4 + 
   27\cdot s^{13}\cdot t^5 + 32\cdot s^{12}\cdot t^6 - 33\cdot s^{11}\cdot t^7 - 39\cdot s^{10}\cdot t^8 + 
   27\cdot s^9\cdot t^9 + 21\cdot s^8\cdot t^{10} - 3\cdot s^7\cdot t^{11} + 2\cdot s^6\cdot t^{12} - 6\cdot s^5\cdot t^{13} - 
   12\cdot s^4\cdot t^{14} + 3\cdot s^2\cdot t^{16} + 2\cdot t^{18})$
   
$y = - s\cdot (2\cdot s^{18} + 3\cdot s^{16}\cdot t^2 - 12\cdot s^{14}\cdot t^4 + 6\cdot s^{13}\cdot t^5 + 2\cdot s^{12}\cdot t^6 + 
   3\cdot s^{11}\cdot t^7 + 21\cdot s^{10}\cdot t^8 - 27\cdot s^9\cdot t^9 - 39\cdot s^8\cdot t^{10} + 
   33\cdot s^7\cdot t^{11} + 32\cdot s^6\cdot t^{12} - 27\cdot s^5\cdot t^{13} - 9\cdot s^4\cdot t^{14} + 
   3\cdot s^3\cdot t^{15} - 3\cdot s^2\cdot t^{16} + 6\cdot s\cdot t^{17} + 2\cdot t^{18})$  
   
$z = s\cdot (-2\cdot s^{18} - 3\cdot s^{16}\cdot t^2 + 12\cdot s^{14}\cdot t^4 + 6\cdot s^{13}\cdot t^5 - 2\cdot s^{12}\cdot t^6 + 
   3\cdot s^{11}\cdot t^7 - 21\cdot s^{10}\cdot t^8 - 27\cdot s^9\cdot t^9 + 39\cdot s^8\cdot t^{10} + 
   33\cdot s^7\cdot t^{11} - 32\cdot s^6\cdot t^{12} - 27\cdot s^5\cdot t^{13} + 9\cdot s^4\cdot t^{14} + 
   3\cdot s^3\cdot t^{15} + 3\cdot s^2\cdot t^{16} + 6\cdot s\cdot t^{17} - 2\cdot t^{18})$
   
$w = - t\cdot (2\cdot s^{18} + 6\cdot s^{17}\cdot t - 3\cdot s^{16}\cdot t^2 + 3\cdot s^{15}\cdot t^3 - 9\cdot s^{14}\cdot t^4 - 
   27\cdot s^{13}\cdot t^5 + 32\cdot s^{12}\cdot t^6 + 33\cdot s^{11}\cdot t^7 - 39\cdot s^{10}\cdot t^8 - 
   27 \cdot s^9\cdot t^9 + 21\cdot s^8\cdot t^{10} + 3\cdot s^7\cdot t^{11} + 2\cdot s^6\cdot t^{12} + 6\cdot s^5\cdot t^{13} - 
   12\cdot s^4\cdot t^{14} + 3\cdot s^2\cdot t^{16} + 2\cdot t^{18})$

\section {The algorithm to compute all rational points of Quartic $X^4 + Y^4 = Z^4 + W^4$} \label{sec:gen}

Consider the formula of the family of curves constructed in this paper:

$(2\cdot m\cdot s^2-2\cdot n\cdot s^2+r\cdot s^3-r\cdot s^2\cdot t+2\cdot m\cdot t^2+2\cdot n\cdot t^2-r\cdot s\cdot t^2+r\cdot t^3)\cdot (2\cdot (s+t)\cdot (m^2-n\cdot r\cdot s + n\cdot r \cdot t))$ = $(-2\cdot m\cdot s^2+2\cdot n\cdot s^2+r\cdot s^3+r\cdot s^2\cdot t+2\cdot m\cdot t^2+2\cdot n\cdot t^2-r\cdot s\cdot t^2-r\cdot t^3)\cdot (2\cdot (s-t)\cdot (n^2-m\cdot r\cdot s - m\cdot r \cdot t))$.
 
We know that this family contains a point that is the image of solution of Leonard Euler for every $s$ and $t$.

We proved that for every rational point from Quartica there exist values $ m,n,s,t$ , so that the image of this point belong to
some cubic curve from the family of cubic curves.

It is easy to see that the formula of family curves contains a variable $r$ only in the first and second degree.

ALGORITHM

(1)Substituting into the formula of the family of curves all possible combinations of values of the variables $ m,n,s,t$
every time we get a second degree equation for $r$.If the discriminant of this equation is a full square we get two pairs of solutions on the family of cubic curves and find two solutions on Quartica. Since the variables $m,n,s,t$ run through all possible integers we cover images of all integer points on Quartica, provided that the discriminant of a quadratic equation for $r$ is full square.If the discriminant equation for $r$ not a full square, we shall now proceed to the second step of the algorithm.

(2)If the discriminant equation for $r$ is not a full square we get two conjugate solution on cubic curve in a expansion of the second degree for some numerical values of $s$ and $t$.But this curve also contains the image of Euler's solution for all values of $s$ and $t$ thus for $s$ and $t$ we selected.So we have three points on the cubic curve: a rational point and two conjugate points of expansion of second degree. We take a line through a rational point and one of the two conjugate points.
And get a new point on the cubic curve in expansion of second degree. 
Our cubic curve contains the point conjugate to the point we constructed. Taking a line through new conjugate points we get a rational point on the curve and then find its image on Quartica.

EXAMPLE.

Consider the family of cubiks constructed in the previous chapter:


$(2\cdot m\cdot s^2-2\cdot n\cdot s^2+r\cdot s^3-r\cdot s^2\cdot t+2\cdot m\cdot t^2+2\cdot n\cdot t^2-r\cdot s\cdot t^2+r\cdot t^3)\cdot (2\cdot (s+t)\cdot (m^2-n\cdot r\cdot s + n\cdot r \cdot t))$ = $(-2\cdot m\cdot s^2+2\cdot n\cdot s^2+r\cdot s^3+r\cdot s^2\cdot t+2\cdot m\cdot t^2+2\cdot n\cdot t^2-r\cdot s\cdot t^2-r\cdot t^3)\cdot (2\cdot (s-t)\cdot (n^2-m\cdot r\cdot s - m\cdot r \cdot t))$.

We substitute:$m=1,n=1,s=1,t=13$,receive equation for $r$:$(r - (13\cdot i/84))\cdot (r + (13\cdot i/84))$,the discriminant equation for $r$ is not a complete square,we receive two solutions of cubic curve in variables $m, n, r$ in a expansion of the second degree :$m=1,n=1,r=(13\cdot i/84)$ and $m=1,n=1,r=-(13\cdot i/84)$.Our first point is:$m_1=1,n_1=1,r_1=(13\cdot i/84)$
To receive the second point we must  get the image of Euler solution for $s=1,t=13$,to solve the system of equations:
$\left\{\begin{array}{clrr} 2\cdot m - 2\cdot n + r - r\cdot 13 + 2\cdot m\cdot 13^2 + 2\cdot n\cdot 13^2 - r\cdot 13^2 + r\cdot 13^3 == 0, \\
 -2\cdot m + 2\cdot n + r + r\cdot 13 + 2\cdot m\cdot 13^2 + 2\cdot n\cdot 13^2 - r\cdot 13^2 - r\cdot 13^3 == 0\end{array} \right.$.

Solving this system we receive the second point on the cubic curve:$m_2=(-92232 /169),n_2=92316 /169,r_2=1$.

Now we  look for a new point on the following cubic curve :$(m_1\cdot k+m_2),(n_1\cdot k+n_2),(r_1\cdot k+r_2)$
When substituting these expressions into the equation of cubic curve we obtain for $k$ the equation of the first degree:

$(69723384192/13 - 7112448\cdot i)  - (733824 - 9766848\cdot i)\cdot k = 0$

Solving this equation we obtain: $k = (7056/169 + 546\cdot i)$

Next, we calculate a new point on the cubic curve in the expansion $i$: $m_3=-504 + 546\cdot i , n_3=588 + 546\cdot i , r_3=(-167/2 + 84\cdot i/13)$.Obviously that point $m_3=-504 - 546\cdot i , n_3=588 - 546\cdot i , r_3=(-167/2 - 84\cdot i/13)$
also belong to the cubic curve. Taking a line through these two points we get a new rational point on the curve.

This is the point : $M=(-2450514024/4855033) , N= 2851182012/4855033$ ,

 $R= (-810875183/9710066)$

If we now calculate a point on the Quartica for a given point on the curve, we get:

               $31557461^4 + 2941868^4 = 1827989^4 + 31557968^4$

the pair of this point is:

                   $324997193816543^4 + 283678931194359^4 =  
                     329177166160259^4 + 277041948785757^4$

\end{document}